%% file: fixed-order-generation.tex
\documentclass[reqno]{amsart}

\usepackage[backref]{hyperref}
\usepackage{mathtools}

\newtheorem{theorem}{Theorem}[section]
\newtheorem{lemma}[theorem]{Lemma}

\newtheorem{corollary}[theorem]{Corollary}

\theoremstyle{definition}
\newtheorem{remark}[theorem]{Remark}

\newtheorem{example}[theorem]{Example}
\newtheorem{conjecture}[theorem]{Conjecture}

\usepackage{amssymb}
\renewcommand{\leq}{\leqslant}
\renewcommand{\geq}{\geqslant}
\newcommand{\eps}{\epsilon}

\newcommand\op[1]{\operatorname{#1}}
\def\P{\mathbf{P}}
\def\E{\mathbf{E}}
\def\Var{\op{Var}}
\def\ord{\op{ord}}

\def\fix{\op{fix}}
\def\calC{\mathcal{C}}

\def\ordn{\ord S_n}
\def\Poisson{\op{Poisson}}

\newcommand\floor[1]{\left\lfloor{#1}\right\rfloor}
\newcommand\pfrac[2]{\left(\frac{#1}{#2}\right)}

\usepackage{todonotes}

\begin{document}

\title{Random generation with cycle type restrictions}
\author{Sean Eberhard}
\address{Sean Eberhard, London, UK}
\email{eberhard.math@gmail.com}

\author{Daniele Garzoni}
\address{Daniele Garzoni, Universit\`a degli Studi di Padova, Dipartimento di Matematica ``Tullio Levi-Civita''}
\email{daniele.garzoni@phd.unipd.it}
\maketitle

\begin{abstract}
We study random generation in the symmetric group when cycle type restrictions are imposed. Given $\pi, \pi' \in S_n$, we prove that $\pi$ and a random conjugate of $\pi'$ are likely to generate at least $A_n$ provided only that $\pi$ and $\pi'$ have not too many fixed points and not too many $2$-cycles. As an application, we investigate the following question: For which positive integers $m$ should we expect two random elements of order $m$ to generate $A_n$? Among other things, we give a positive answer for any $m$ having any divisor $d$ in the range $3 \leq d \leq o(n^{1/2})$.
\end{abstract}

\section{Introduction}

Fix two conjugacy classes $\calC, \calC'$ of the symmetric group $S_n$. Should we expect random elements $\pi\in\calC, \pi'\in\calC'$ to generate at least $A_n$?

Clearly not if $\calC$ and $\calC'$ both represent elements with many fixed points, for then it is likely that $\pi$ and $\pi'$ have a common fixed point, and therefore cannot possibly generate $A_n$. Also not if $\calC$ and $\calC'$ both represent elements with many $2$-cycles: if both $\calC$ and $\calC'$ represent elements with $\Omega(n)$ $2$-cycles, then it is not hard to see that the probability that $\pi$ and $\pi'$ have a common $2$-cycle is bounded away from zero.

The purpose of this paper is to prove that, apart from these two basic obstructions, yes we should expect $\pi, \pi'$ to generate at least $A_n$. The following is our main theorem.
\begin{theorem}\label{main-thm}
Let $\calC, \calC' \subset S_n$ be fixed conjugacy classes. For each $j$ let $c_j$ be the number of $j$-cycles in a representative element of $\calC$, and similarly $c'_j$ for $\calC'$. Let $\pi\in\calC, \pi'\in\calC'$ be chosen uniformly at random. Assume that $c_1, c'_1 = o(n^{2/3})$. Then

\begin{enumerate}
    \item $\P(\langle \pi, \pi' \rangle \geq A_n) = 1 - o(1)$ if and only if
    \[
      (c_1 + c_2^{1/2}) ({c'_1} + {c'_2}^{1/2}) = o(n) ;
    \]
    \item $\P(\langle \pi, \pi' \rangle \geq A_n) = \Omega(1)$ if and only if
    \begin{align*}
      &(c_1 + c_2^{1/2}) ({c'_1} + {c'_2}^{1/2}) = O(n), ~\text{and}\\
      &c_2 + c'_2 = n - \Omega(n).
    \end{align*}
\end{enumerate}
In fact the hypothesis $c_1, c'_1 = o(n^{2/3})$ is not needed for the ``only if'' statements.
\end{theorem}

The point of the theorem is that the likelihood of $\langle \pi, {\pi'}^\sigma\rangle \geq A_n$ is more-or-less completely characterized by counting fixed points and $2$-cycles. The symmetric version of the theorem is illustrative:

\begin{corollary} \label{cor:symmetric-case}
Assume $c_1 = c'_1$ and $c_2 = c'_2$.
\begin{enumerate}
    \item $\P(\langle \pi, \pi' \rangle \geq A_n) = 1-o(1)$ if and only if $c_1 = o(n^{1/2})$ and $c_2 = o(n)$.
    \item $\P(\langle \pi, \pi' \rangle \geq A_n) = \Omega(1)$ if and only if $c_1 = O(n^{1/2})$ and $c_2 = n/2 - \Omega(n)$.
\end{enumerate}
\end{corollary}

Our theorem refines some previous results about random generation in $S_n$. The most well-known and basic of these is Dixon's theorem~\cite{dixon}, confirming a conjecture of Neto, that two random elements of $S_n$ almost surely generate at least $A_n$. Confirming a conjecture of Robinson, Shalev~\cite{shalev} proved that this remains true if we require the two elements to be conjugate. One can view Theorem~\ref{main-thm} as a common generalization of all these two results, since a random permutation almost surely has very few fixed points and very few $2$-cycles.\footnote{To be precise, for any $\omega_n \to \infty$, the probability that $\pi$ has at most $\omega_n$ fixed points tends to $1$ as $n\to\infty$, and similarly for $2$-cycles.} In another direction, Babai and Hayes~\cite{babai_hayes} proved that if $\pi$ is fixed and $\pi'$ is uniformly random, then $\pi$ and $\pi'$ will almost surely generate at least $A_n$ if and only if $\pi$ has $o(n)$ fixed points. This does not follow from our theorem (because of the $c_1, c'_1 = o(n^{2/3})$ restriction), but it certainly has a similar flavour.

Our original motivation was the following question. Let $\ordn$ be the set of all $m$ which occur as the order of at least one $\pi \in S_n$. Fix $m \in \ord S_n$, and let $\pi$ and $\pi'$ be two random elements of order $m$. Should we expect $\pi$ and $\pi'$ to generate at least $A_n$? For the almost-sure version of the question, we give a positive answer for a broad class of $m$. Specifically, we prove the following theorem.

\begin{theorem} \label{thm:fixed-order-almost-sure}
Let $m \in \ord S_n$, and assume that either
\begin{enumerate}
    \item $m$ has a divisor $d$ in the range $3 \leq d \leq o(n^{1/2})$, or
    \item $m$ is even and there is at least one $\pi \in S_n$ with $o(n^{1/2})$ fixed points and $o(n)$ $2$-cycles.
\end{enumerate}
Then two random elements of $S_n$ of order $m$ almost surely generate at least $A_n$.
\end{theorem}

For the positive-probability version of the question, we give a complete characterization.

\begin{theorem} \label{thm:fixed-order-positive-prob}
Let $m \in \ord S_n$. Then two random elements of order $m$ generate at least $A_n$ with probability bounded away from zero if and only if
\begin{enumerate}
    \item $m$ is odd and there is at least one $\pi \in S_n$ of order $m$ with $O(n^{1/2})$ fixed points, or
    \item $m$ is even and not $2$.
\end{enumerate}
\end{theorem}

The novelty of these results is the wide range we allow for $m$. The case of bounded $m$ is included in work of Liebeck and Shalev: see~\cite{liebeck-shalev-23generation}, and particularly \cite{liebeck-shalev-fuchsian-groups}. In the latter paper the authors consider more generally random homomorphisms $\Gamma \to S_n$ for any Fuchsian group $\Gamma$: thus for example they prove that $A_n$ is generated by random elements satisfying $x^2 = y^3 = (xy)^7 = 1$. It is clear that our rather elementary approach, relying essentially on the independence of the generators, does not apply to such a problem. On the other hand we are not aware of any previous results for unbounded $m$.

\subsection{Notation and idiosyncrasies}

Much of our notation has already appeared. We consistently write $\pi$ and $\pi'$ for our two permutations, which we assume are random subject to having $c_j$ and $c'_j$ $j$-cycles, respectively, for each $j$. We will write $c$ and $c'$ for the total number of cycles:
\begin{align*}
  c&= c_1 + \cdots + c_n \\
  c'&= c'_1 + \cdots + c'_n.
\end{align*}
Occasionally, we will denote cycle types using exponent notation: thus for example the conjugacy class $\calC$ consists of all permutations with cycle type $(1^{c_1}, 2^{c_2}, \dots, n^{c_n})$.

We write $\Omega$ for our ground set of size $n$, and we write $S_n$ and $A_n$ for the symmetric and alternating groups on $\Omega$. If $G$ acts on a set $X$, $g\in G$, $x \in X$, we write $x^g$ for the image of $x$ under $g$. In particular, if $g, h \in G$ we write $h^g$ for $g^{-1}hg$. Thus if $\fix_X(g)$ is the fixed point set of $g$ then we have
\[
  \fix_X(g^h) = \fix_X(g)^h.
\]

We use standard big-O and little-o without reservation, all understood with respect to the limit $n \to\infty$. Additionally we use the Vinogradov notation $x \ll y$ to mean $x = O(y)$, and $x \asymp y$ to mean $x \ll y$ and $x \gg y$ (i.e., $x/y$ is bounded above and below by positive constants).

As will be clear already and annoying to some readers, we are mainly interested in asymptotic properties of $S_n$ as $n\to\infty$, so almost all of our results are expressed with inexplicit error terms like $o(1)$ or $1 - o(1)$. None of our methods are especially advanced, so it is certainly possible to prove everything with explicit error terms, but doing so would be notationally cumbersome and, we feel, distracting. Instead, we have prioritized statements which are as general as possible with respect to the conjugacy classes $\calC$ and $\calC'$. If ever our theorem is needed with an explicit error term for specific conjugacy classes, it will be easy enough to rehash the essential arguments.

In line with this philosophy, we do not hesitate to say ``almost surely'', when really we mean ``with probability tending to $1$ as $n\to\infty$'', and ``with positive probability'' instead of ``with probability bounded away from $0$'', etc. This technically imprecise language could actually be made precise by working with an ultraproduct of $(S_n)_{n=1}^\infty$ instead of the finite groups $S_n$, but there is no reader whom this would help.

\subsection{Organization of the paper}

We now briefly sketch the proof of Theorem~\ref{main-thm}, and explain the organization of the paper.

First, we prove the ``only if'' statements in Theorem~\ref{main-thm} in Section~\ref{sec:converses}. We also give a few examples justifying the $c_1, c'_1 = o(n^{2/3})$ hypothesis. We hope that starting with a wealth of examples will orient the reader, and motivate the rest of the paper.

In particular, it will be clear from these examples that the main obstruction in Theorem~\ref{main-thm} is transitivity. Therefore in Section~\ref{sec:intransitive} we estimate the probability that $G = \langle \pi, \pi'\rangle$ is transitive by studying the distribution of the number $N$ of orbits of $G$ of size less than $n/2$. This is the most technical part of the paper. We will prove that if
\[
  (c_1 + c_2^{1/2}) (c'_1 + {c'_2}^{1/2}) = o(n)
\]
then $\E N = o(1)$, and hence $G$ is almost surely transitive. Otherwise, if
\[
  (c_1 + c_2^{1/2}) (c'_1 + {c'_2}^{1/2}) = O(n)
\]
and
\[
  c_2 + c'_2 = n - \Omega(n)
\]
then we will prove $\E N = O(1)$. We will then use the method of moments to prove the Poisson-type approximation
\[
  \P(N = 0) = e^{-\E N} + o(1).
\]
Hence $G$ is transitive with probability bounded away from zero.

If $G$ is transitive, then very likely $G \geq A_n$. We prove this in Section~\ref{sec:transitive-subgroups}. We treat the imprimitive and primitive cases separately. Neither argument is very difficult, and the bounds are much stronger than the bounds we give for the intransitive case. For the primitive case we use results depending the classification of finite simple groups. This will complete the proof of Theorem~\ref{main-thm}.

Finally, in Section~\ref{sec:fixed-order} we study fixed-order generation. Again we give a wealth of examples, and we prove Theorems~\ref{thm:fixed-order-almost-sure} and \ref{thm:fixed-order-positive-prob} using Theorem~\ref{main-thm}.

\section{Necessity of the hypotheses} \label{sec:converses}

Our first order of business is to discuss the necessity of the various hypotheses in Theorem~\ref{main-thm}. We prove the stated ``only if'' conditions, and additionally we discuss some examples where $c_1 + c'_1$ exceeds $n^{2/3}$.

\subsection{The case \texorpdfstring{$c_1 c'_1/n = \Omega(1)$}{c1 c1' / n = Omega(1)} or \texorpdfstring{$c_1 c'_1/n = \omega(1)$}{c1 c1' = omega(1)}} \label{subsec:fixed-points}

If $\pi$ and $\pi'$ each have many fixed points then it is likely that they have a common fixed point. The expected number of common fixed points\ of $\pi$ and $\pi'$ is exactly $c_1 c'_1 / n$. We claim that if $c_1 c'_1 / n = \Omega(1)$ then there is a nonvanishing probability that $\pi$ and $\pi'$ have a common fixed point, and if $c_1 c'_1 / n = \omega(1)$ then almost surely $\pi$ and $\pi'$ have a common fixed point. To see this, note that $\fix(\pi)$ is just a random set of size $c_1$ and $\fix(\pi')$ is just a random set of size $c'_1$. The probability that two random sets of size $c_1$ and $c'_1$ are disjoint is
\begin{align*}
  \frac{n-c_1}{n} \frac{n-c_1-1}{n-1} \cdots \frac{n-c_1-c'_1+1}{n-c'_1+1}
  &\leq \exp\left(
    - \frac{c_1}{n} - \frac{c_1}{n-1} - \cdots - \frac{c_1}{n-c'_1+1}
  \right) \\
  &\leq \exp\left(
    - \frac{c_1 c'_1}{n}
  \right).
\end{align*}
This proves both our claims.

\subsection{The case \texorpdfstring{$c_2 c'_2 \gg n^2$}{c2 c2' >> n\textasciicircum2}}

Suppose that $c_2 c'_2 \gg n^2$. We claim that there is a nonvanishing probability that $\pi$ and $\pi'$ have a common $2$-cycle. To see this, fix the $c_2$ $2$-cycles of $\pi$, and consider picking the $c'_2$ $2$-cycles of $\pi'$ in order. The $k$-th $2$-cycle is chosen uniformly from a pool of $\binom{n-2k+2}{2}$ possibilities, of which at least $c_2-2k+2$ are also $2$-cycles of $\pi$. Thus the probability that each of these fails to be a common $2$-cycle is bounded by
\[
\prod_{k \leq c'_2, c_2/2} \left( 1- \frac{c_2-2k+2}{\binom{n-2k+2}{2}} \right) \leq \exp \left(-\Omega \left(\frac{c_2 c'_2}{n^2}\right)\right).
\]
If $c_2 c'_2 \gg n^2$ then this is bounded away from 1.

\subsection{The case \texorpdfstring{$c_1 {c'_2}^{1/2}/n = \Omega(1)$}{c1 c2'\textasciicircum{1/2} / n = Omega(1)} or \texorpdfstring{$c_1 {c'_2}^{1/2}/n = \omega(1)$}{c1 c2'\textasciicircum{1/2} / n = omega(1)}, or vice versa} \label{subsec:c1c2}

Morally, a similar argument applies if $c_1^2 c'_2 \gg n^2$, but the situation is complicated by the need to track two ``state variables'': the number of remaining $2$-cycles to place, and the number of exhausted fixed points. Instead we use moment methods.

Assume $c_1^2 c'_2 \gg n^2$, and let $N$ be the number of $2$-cycles $(xy)$ of $\pi'$ such that $x,y$ are both fixed points of $\pi$. Clearly,
\[
\E N = \frac{c_1(c_1-1)c'_2}{n(n-1)}=\frac {c_1^2 c'_2}{n^2}(1+O(1/c_1)).
\]
Counting ordered pairs of such $2$-cycles, we have
\begin{align*}
  \E(N(N-1))
  &= \frac{c_1(c_1-1)(c_1-2)(c_1-3)c'_2(c'_2-1)}{n(n-1)(n-2)(n-3)} \\
  &= \pfrac{c_1^2 c'_2}{n^2}^2 (1 + O(1/c_1 + 1/c'_2)).
\end{align*}
Hence
\[
  \Var N = \E N + \E(N(N-1)) - (\E N)^2 \leq \E N + O(1/c_1 + 1/c'_2) (\E N)^2,
\]
so
\[
  \P(N=0) \leq \frac{\Var N}{(\E N)^2} = \frac1{\E N} + O(1/c_1 + 1/c'_2).
\]
Note now that if $c_1^2 c'_2 \gg n^2$ then $c_1 \gg n^{1/2}$. The case in which $c_1 \asymp n$ and $c'_2 \asymp 1$ can be handled separately, so assume also $c'_2 = \omega(1)$. Then we have
\begin{align*}
  \P(N=0) &\leq \frac1{\E N} + o(1),~\text{where}\\
  \E N &\sim \frac{c_1^2 c'_2}{n^2}.
\end{align*}
Thus if $c_1^2 c'_2 / n^2 = \omega(1)$ then $\P(N=0) = o(1)$.

The case in which $c_1^2 c'_2 / n^2 \asymp 1$ is a little more delicate, but can be handled with the method of moments. By a similar argument we have
\[
  \E (N (N-1) \cdots (N-k+1)) = \pfrac{c_1^2 c'_2}{n^2}^k (1 + O_k(1/c_1 + 1/c'_2)).
\]
Moreover,
\[
  \E (N(N-1) \cdots (N-k+1)) \leq \pfrac{c_1^2 c'_2}{n^2}^k.
\]
Thus from Bonferroni's inequalities we have
\begin{align*}
  \P(N = 0)
  &= \sum_{k=0}^{K-1} (-1)^k \E \binom{N}{k} + O\left( \E \binom{N}{K} \right) \\
  &= \sum_{k=0}^{K-1} \frac{(-1)^k}{k!} \pfrac{c_1^2 c'_2}{n^2}^k (1 + O_k(1/c_1 + 1/c'_2)) + O\left( \frac1{K!} \pfrac{c_1^2 c'_2}{n^2}^K \right) \\
  &= e^{-c_1^2 c'_2 / n^2} + O_K \left( 1/c_1 + 1/c'_2 \right) + O \left(\frac1{K!} \pfrac{c_1^2 c'_2}{n^2}^K \right).
\end{align*}
Assuming $c_1^2 c'_2 / n^2 \asymp 1$, we can choose $K=O_\eps(1)$ so that the second error term is smaller than $\eps$, and then as before the first error term is $o_\eps(1)$. Hence
\[
  \P(N=0) = e^{-c_1^2 c'_2 / n^2} + o(1).
\]
In particular, $\P(N=0)$ is bounded away from $1$.

We will use a similar argument later, in Subsection~\ref{subsec:possion}, to prove the positive direction of Theorem~\ref{main-thm}(2).

\subsection{The case \texorpdfstring{$c_2, c'_2 = n/2 - o(n)$}{c2, c2' = n/2 - o(n)}}

Now suppose that $c_2, c'_2 = n/2 - o(n)$. We have already shown that the probability that $\pi$ and $\pi'$ have a common $2$-cycle is bounded away from zero. We claim further that $\langle \pi, \pi'\rangle$ is almost surely intransitive. We do not claim however that $\pi$ and $\pi'$ almost surely have a common $2$-cycle -- instead we use a more involved argument.

Let $N$ be the number of sets $X \subset \Omega$ with the following properties:
\begin{enumerate}
    \item $0 < |X| < n^{1/3}$,
    \item $X$ is preserved by both $\pi$ and $\pi'$,
    \item $\pi|_X$ and $\pi'|_X$ both consist only of $2$-cycles, and
    \item $\langle \pi, \pi'\rangle |_X$ is transitive.
\end{enumerate}
In other words, $N$ is the number of orbits of $\langle \pi, \pi'\rangle$, of size less than $n^{1/3}$, consisting exclusively of $2$-cycles. We claim that $N$ is almost surely positive, so in particular $\langle \pi, \pi'\rangle$ is almost surely not transitive.

To prove this we use the second moment method. Let $N_k$ be the number of such sets of size $2k$. Then
\[
  \E N_k = \binom{n}{2k}^{-1} \binom{c_2}{k} \binom{c'_2}{k} p(k),
\]
where $p(k)$ is the probability that two elements of $S_{2k}$, each of cycle type $(2^{k/2})$, generate a transitive subgroup. Using $k = o(n^{1/2})$ and Stirling's approximation,
\[
  \E N_k
  \sim \pfrac{c_2 c'_2}{n^2}^k \frac{(2k)!}{k!^2} p(k)
  \asymp \pfrac{4 c_2 c'_2}{n^2}^k \frac{p(k)}{k^{1/2}}.
\]
Next we need to estimate $p(k)$.

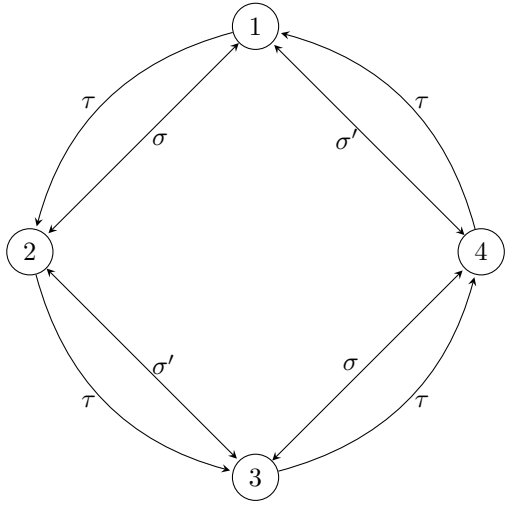
\begin{figure}
    \centering
    \input{tikz-bijection.tex}
    \caption{Bijection between pairs of permutations $\sigma, \sigma'$ of cycle type $(2^k)$ and permutations $\tau$ having only even-length cycles}
    \label{fig:bijection}
\end{figure}

\begin{lemma} \label{lem:p(k)}
$ p(k) \asymp k^{-1/2}. $
\end{lemma}
\begin{proof}
There is a bijection between ordered pairs of permutations $\sigma, \sigma' \in S_{2k}$ of cycle type $(2^k)$ and permutations $\tau \in S_{2k}$ having only even-length cycles, defined as follows (see Figure~\ref{fig:bijection}). Given $\sigma, \sigma'$, define $\tau$ by repeating the following process: find the smallest (in some arbitrary but fixed order) element $x$ at which we have not yet defined $\tau$, and define
\[
  x
  \xrightarrow\tau x^\sigma
  \xrightarrow\tau x^{\sigma \sigma'}
  \xrightarrow\tau x^{\sigma \sigma' \sigma}
  \xrightarrow\tau \cdots
  \xrightarrow\tau x.
\]
It is easy to see that the cycle must terminate after an even number of steps, so the process does indeed define a permutation having only even-length cycles. In the other direction, suppose $\tau \in S_{2k}$ has only even-length cycles. On a given cycle of $\tau$, find the smallest element $x$ and define
\[
  x
  \xleftrightarrow{\sigma} x^{\tau}
  \xleftrightarrow{\sigma'} x^{\tau^2}
  \xleftrightarrow{\sigma} x^{\tau^3}
  \xleftrightarrow{\sigma'} \cdots
  \xleftrightarrow{\sigma'} x.
\]
These maps are clearly mutually inverse.

Now note that $\langle \sigma, \sigma'\rangle$ is transitive if and only if $\tau$ is a $2k$-cycle. Thus
\begin{align*}
  p(k)
  &= \P(\langle \sigma, \sigma'\rangle~\text{transitive}) \\
  &= \frac{|\{\tau: \tau~\text{is a $2k$-cycle}\}|}{|\{\sigma: \sigma~\text{has cycle type}~(2^k)\}|^2} \\
  &= \frac{(2k)!/(2k)}{((2k)!/(2^k k!))^2} \\
  &= \frac{4^k k!^2}{2k (2k)!} \\
  &\sim \frac{\sqrt{\pi}}{2} k^{-1/2}.
\end{align*}
In the last line we applied Stirling's approximation.
\end{proof}

Hence
\[
  \E N_k \asymp \pfrac{4 c_2 c'_2}{n^2}^k \frac1{k},
\]
and
\[
  \E N = \sum_{1 \leq k < n^{1/3}} \E N_k \asymp \sum_{1 \leq k < n^{1/3}} \pfrac{4 c_2 c'_2}{n^2}^k \frac1k.
\]
As long as $c_2, c'_2 = n/2 - o(n)$, this is $\omega(1)$.

In order to estimate $\Var N$, consider $\E(N(N-1))$. This is the expected number of ordered pairs of distinct sets $X, Y$ satisfying conditions (1) through (4). Let $E_X$ be the event that $X$ satisfies (1) through (4). If $|X| = 2k$, we have
\[
  \P(E_X) = \frac{\binom{c_2}{k} \binom{c'_2}{k}}{\binom{n}{2k}^2} p(k).
\]
Because of condition (4), $E_X \cap E_Y = \emptyset$ unless $X=Y$ or $X \cap Y = \emptyset$. If $X \cap Y = \emptyset$, and $|X| = 2k$ and $|Y| = 2l$ say, then
\begin{align*}
  \P(E_X \cap E_Y)
  &= \frac{\binom{c_2}{k,l,c_2-k-l} \binom{c'_2}{k,l,c'_2-k-l}}{\binom{n}{2k, 2l, n-2k-2l}^2} p(k) p(l) \\
  &= \P(E_X) \P(E_Y) (1 + O((k+l)^2 / n)).
\end{align*}
It follows from this that
\[
  \E(N(N-1)) \leq (\E N)^2 (1 + O(n^{-1/3})),
\]
and so
\[
  \Var N = \E N + \E(N(N-1)) - (\E N)^2 \leq \E N + O(n^{-1/3}) (\E N)^2.
\]
Thus
\[
  \P(N = 0) \leq \frac{\Var N}{\E N^2} \leq \frac1{\E N} + O(n^{-1/3}).
\]
Since $\E N = \omega(1)$, this implies $N > 0$ almost surely.

\subsection{Cases in which \texorpdfstring{$c_1 + c'_1 \gg n^{2/3}$}{c1 + c1' >> n\textasciicircum{2/3}}}

We now discuss some examples in which $c_1 + c'_1 \gg n^{2/3}$, say $c_1 \gg n^{2/3}$. First, suppose that $c_1^3 c'_3 \gg n^3$. In this case a straightforward modification of our argument in the $c_1^2 c'_2 \gg n^2$ case (Subsection~\ref{subsec:c1c2}) shows that there is likely a $3$-cycle $(xyz)$ of $\pi'$ such that $x,y,z$ are fixed points of $\pi$. Similarly if $c_1^4 c'_4 \gg n^4$, etc.

These examples are not too troubling. In fact, a slight elaboration of our proof shows that if $c_1 + c'_1 \leq n^{1-\Omega(1)}$ then it is enough to assume
\[
  (c_1 + c_2^{1/2} + \cdots + c_n^{1/n})
  (c'_1 + {c'_2}^{1/2} + \cdots + {c'_n}^{1/n}) = o(n)
\]
for the almost-sure problem, and
\begin{align*}
  &(c_1 + c_2^{1/2} + \cdots + c_n^{1/n})
  (c'_1 + {c'_2}^{1/2} + \cdots + {c'_n}^{1/n}) = O(n) \\
  &c_2 + c'_2 = n - \Omega(n)
\end{align*}
for the positive-probability problem. (Note that if $c_1 + c'_1 \leq n^{1-\Omega(1)}$ then really only boundedly many terms in these expressions are important.)

The following example is more troubling. Suppose $\pi$ is a random $n$-cycle and $\pi'$ a random transposition. Then up to conjugacy $\pi$ is $(1 2 \cdots n)$ and $\pi'$ is $(x n)$ for some $x$ chosen uniformly from $\{1, \dots, n-1\}$. Let $G = \langle \pi, \pi'\rangle$. Then $G$ is determined by $\gcd(x, n)$: if $\gcd(x, n)>1$ then $G$ is imprimitive, while if $\gcd(x, n)=1$ then $G = S_n$. Thus the probability that $G \geq A_n$ is $\varphi(n) / (n-1)$, which can be small or large depending on the arithmetic of $n$. This example illustrates the difficulty of pinning down a pithy if-and-only-if condition without any a priori hypothesis on $c_1 + c'_1$.

Finally, we note that if $c + c' > n + 1$ then there are in fact no $\pi \in \calC$ and $\pi'\in\calC'$ such that $\langle \pi, \pi'\rangle$ is transitive.

\section{Intransitive subgroups} \label{sec:intransitive}

Let $G = \langle \pi, \pi'\rangle$. As so often the case in random generation problems in $S_n$, the main obstruction is transitivity: if $G$ is transitive then very likely $G \geq A_n$ (we will prove this in Section~\ref{sec:transitive-subgroups}). In other words, our main job is to estimate the probability that $\pi$ and $\pi'$ have a common fixed set of some size $k$.

Let $N_k$ be the number of orbits of $G$ of size $k$, in other words the number of $k$-sets fixed by $G$ and on which $G$ acts transitively. Let
\[
  N = \sum_{1\leq k \leq n/2} N_k.
\]
Our goal is to estimate $\P(N = 0)$. We will accomplish this in two stages:
\begin{enumerate}
    \item We will bound $\lambda = \E N$. This will already establish what we need for the almost-sure problem, since $\P(N = 0) \geq 1 - \lambda$.
    \item We will prove a Poisson-type approximation for $N$ which shows that $\P(N=0) \approx e^{-\lambda}$.
\end{enumerate}

\subsection{Bounding \texorpdfstring{$\lambda = \E N$}{lambda = EN}}

There are many ways we can form a fixed $k$-set from the cycles of $\pi$ and $\pi'$. For any two partitions
\begin{align}
    k &= d_1 1 + d_2 2 + \cdots + d_k k, \label{eqn:d-partition}\\
    k &= d'_1 1 + d'_22 + \cdots + d'_k k, \label{eqn:d'-partition}
\end{align}
such that $d_j \leq c_j, d'_j \leq c'_j$ for each $j$, we can form a fixed $k$-set by taking $d_j$ $j$-cycles from $\pi$ and $d'_j$ $j$-cycles from $\pi'$, for each $j$. The probability that these $k$-sets coincide after $\pi'$ is conjugated by $\sigma$ is $1 / \binom{n}{k}$. Thus we have
\begin{equation} \label{eqn:ENk-formula}
  \E N_k = \sum_{\eqref{eqn:d-partition}, \eqref{eqn:d'-partition}} \frac{\binom{c_1}{d_1} \cdots \binom{c_k}{d_k} \binom{c'_1}{d'_1} \cdots \binom{c'_k}{d'_k}}{\binom{n}{k}} p(d_1, \dots, d_k ; d'_1, \dots, d'_k),
\end{equation}
where $p(d_1, \dots, d_k ; d'_1, \dots, d'_k)$ denotes the probability that random permutations $\tau, \tau' \in S_k$ with cycle types $(1^{d_1}, \dots, k^{d_k})$ and $(1^{d'_1}, \dots, k^{d'_k})$ generate a transitive subgroup.\footnote{Of course, this is the very thing we are trying to estimate in this section. This ``self-similarity'' is largely incidental: when we write $p(d_1, \dots, d_k ; d'_1, \dots, d'_k)$ we usually have in mind $k, d_1, \dots, d'_k$ of bounded size, whereas in the bigger picture of this section we are concerned with asymptotics.}

In this subsection we simply ignore $p(d_1, \dots, d_k ; d'_1, \dots, d'_k)$, bounding it by $1$. Thus from \eqref{eqn:ENk-formula} we have simply
\[
  \E N_k \leq \frac{ f_k f'_k }{\binom{n}{k}},
\]
where
\begin{equation}\label{fk}
  f_k = \sum_{\eqref{eqn:d-partition}} \binom{c_1}{d_1} \cdots \binom{c_k}{d_k}
\end{equation}
is the number of $k$-sets fixed by $\pi$, and similarly $f'_k$. (In the next subsection we will need the full force of \eqref{eqn:ENk-formula}.)

\begin{lemma} \label{lem:fkx-bound}
We have
\[
  f_k \leq \min_{x > 0}\, x^{-k} (1+x)^{c_1} (1+x^2)^{c_2} \cdots (1+x^k)^{c_k}.
\]
\end{lemma}
\begin{proof}
This follows from
\begin{align*}
  f_k x^k
  &= \sum_{\eqref{eqn:d-partition}} \binom{c_1}{d_1} \cdots \binom{c_k}{d_k} x^{d_1 1 + \cdots + d_k k} \\
  &\leq \sum_{d_1, \dots, d_k\geq 0} \binom{c_1}{d_1} \cdots \binom{c_k}{d_k} x^{d_1 1 + \cdots + d_k k} \\
  &= (1+x)^{c_1} (1+x^2)^{c_2} \cdots (1+x^k)^{c_k}.\qedhere
\end{align*}
\end{proof}

Define the function $h$ by
\[
  h(x) = x \log \frac1x + (1-x) \log\frac1{1-x}.
\]
It is well known that
\[
  k^{-1/2} e^{h(k/n)n} \ll \binom{n}{k} \leq e^{h(k/n)n} \qquad (1\leq k\leq n/2).
\]
The latter inequality follows from
\[
  \binom{n}{k} x^k \leq (1+x)^n,
\]
with $x = k/(n-k)$. The former inequality follows from Stirling's approximation.

\begin{lemma} \label{fk-bound}
For any $\eps \in (0, 1/2)$ we have
\[
  f_k \leq
  \pfrac{c_1}{k}^k e^{O(\eps^{-2}k)} +
  e^{h\pfrac{k}{2c_2} c_2 + \eps k} +
  \pfrac{n}{k}^{k/3} e^{O(\eps^{-2}k)}
  .
\]
\end{lemma}

Note that the middle term here is of the form
\[
  \pfrac{c_2}{k}^{k/2} e^{O(k)}.
\]
The more precise bound is important to us, however.

\begin{proof}
First note that we may assume $k \leq 3c/2$, for if $k \geq 3c/2$ then directly from \eqref{fk} we have
\[
  f_k \leq 2^c \leq 2^{2k/3},
\]
which is subsumed by the third term in the claim.

Now assume $k \leq 3c/2$. By the previous lemma we have, if $x \leq 1$,
\begin{equation} \label{eqn:fkx-bound}
  f_k \leq x^{-k} (1+x)^{c_1} (1+x^2)^{c_2} (1+x^3)^{c}.
\end{equation}
The unique minimizer of this expression satisfies
\[
-\frac{k}{x} + \frac{c_1}{1+x} + \frac{2c_2 x}{1+x^2} + \frac{3c x^2}{1+x^3} = 0,
\]
and hence
\[
  c_1 \pfrac{x}{1+x} + 2c_2 \pfrac{x^2}{1+x^2} + 3c \pfrac{x^3}{1+x^3} = k.
\]
The solution to this equation satisfies $x \leq 1$ (since $3c/2 \geq k$), and is comparable to the smallest of $k/c_1$, $(k/c_2)^{1/2}$, $(k/c)^{1/3}$. We consider these three cases separately. More precisely, because we need to take extra care with the $c_2$ dependence, we consider three cases depending on which of
\[
  c_1/k,\quad \frac12 \eps (c_2 / k)^{1/2},\quad (c / k)^{1/3}
\]
is largest.
\begin{enumerate}
    \item First suppose $c_1 / k \geq \frac12 \eps (c_2 / k)^{1/2}, (c / k)^{1/3}$. Take $x = k / c_1$ in \eqref{eqn:fkx-bound}. The result is
    \[
      f_k \leq \pfrac{c_1}{k}^k e^{\pfrac{k}{c_1} c_1 + \pfrac{k}{c_1}^2 c_2 + \pfrac{k}{c_1}^3 c} \leq \pfrac{c_1}{k}^k e^{O(\eps^{-2} k)}.
    \]
    \item Next suppose $\frac12 \eps(c_2 / k)^{1/2} \geq c_1/k, (c/k)^{1/3}$. Note the latter condition implies $k \leq \eps^6 c_2 \leq c_2 / 2$. Take
    \[
      x = \pfrac{k}{2c_2 - k}^{1/2}
    \]
    in \eqref{eqn:fkx-bound}. The result is
    \[
      f_k
      \leq e^{h \pfrac{k}{2c_2} c_2 + \pfrac{k}{2c_2-k}^{1/2} c_1 + \pfrac{k}{2c_2-k}^{3/2} c}
      \leq e^{h \pfrac{k}{2c_2} c_2 + \eps k}.
    \]
    \item Finally, suppose $(c / k)^{1/3} \geq \frac12 \eps (c_2 / k)^{1/2}, c_1 / k$. Take $x = (k/c)^{1/3}$ in \eqref{eqn:fkx-bound}. The result is
    \[
      f_k \leq \pfrac{c}{k}^{k/3} e^{\pfrac{k}{c}^{1/3} c_1 + \pfrac{k}{c}^{2/3} c_2 + \pfrac{k}{c} c} \leq \pfrac{c}{k}^{k/3} e^{O(\eps^{-2} k)}.
    \]
\end{enumerate}
Since $c \leq n$ this proves the lemma.
\end{proof}

\begin{lemma} \label{lem:Nk-bound-small-k}
We have
\[
  \E N_k \ll \left(
  \frac{(c_1 + c_2^{1/2}) (c'_1 + {c'_2}^{1/2})}{n} + \frac{c_1 + c'_1}{n^{2/3}} + \pfrac{k}{n}^{1/6}
  \right)^k e^{O(k)}.
\]
Moreover, if $c_2 + c'_2 = n - \Omega(n)$ then
\begin{align*}
  \E N_k
  &\ll \frac1{k^{k/3}}
  \left(
  \frac{ c_1 c'_1 + c_1 {c'_2}^{1/2} + {c_2}^{1/2} c'_1}{n} + \frac{c_1 + c'_1}{n^{2/3}}
  \right)^k O(1)^k \\
  &\qquad + (1 - \Omega(1))^k + O(k/n)^{k/6}.
\end{align*}
\end{lemma}
\begin{proof}
We have
\[
  \E N_k \leq f_k f'_k / \binom{n}{k}.
\]
Bounding each of $f_k$ and $f'_k$ using the previous lemma and expanding the product, we get
\begin{align*}
    \E N_k 
    &\ll \frac1{k^k} \pfrac{c_1 c'_1}{n}^k e^{O(\eps^{-2} k)} \\
    &\qquad + \frac1{k^{k/2}} \pfrac{c_1 {c_2'}^{1/2} + {c_2}^{1/2} c'_1}{n}^k e^{O(\eps^{-2} k)} \\
    &\qquad + \frac1{k^{k/3}} \left(\frac{c_1 + c'_1}{n^{2/3}} \right)^k e^{O(\eps^{-2} k)} \\
    &\qquad + \left(\frac{k}{n} \right)^{k/6} e^{O(\eps^{-2} k)} \\
    &\qquad + k^{1/2} e^{h\left(\frac{k}{2c_2}\right) c_2 + h\left(\frac{k}{2c_2'}\right)c'_2 - h(k/n)n + 2 \eps k}.
\end{align*}
Here we used the simple bound $\binom{n}{k} \geq (n/k)^k$ for every term except the one involving both $c_2$ and $c'_2$, where we used the more precise bound $\binom{n}{k} \gg k^{-1/2} e^{h(k/n) n}$. Additionally, for the other terms involving $c_2$ or $c'_2$ we used
\[
  e^{h\pfrac{k}{2c_2}c_2} \leq \pfrac{c_2}{k}^{k/2} e^{O(k)}.
\]
If we apply this bound also to the $c_2, c'_2$ term, and bound various things crudely, then we get the first statement in the lemma.

To get the second statement, the only term which requires further comment is the $c_2, c'_2$ term. Since $h$ is concave and $h'(x) = \log(1/x - 1)$, we have
\[
  h(x) \leq h(k/n) + (x - k/n) \log(n/k - 1).
\]
Hence, by a short calculation,
\begin{align*}
  h\pfrac{k}{2c_2} c_2 + h\pfrac{k}{2c'_2} c'_2 - h\pfrac{k}{n} n
  &\leq -(n - c_2 - c'_2) \log \pfrac{n}{n-k} \\
  &\leq -(n - c_2 - c'_2) k/n.
\end{align*}
Assuming $c_2 + c'_2 = n - \Omega(n)$, this is $-\Omega(k)$, so the lemma follows from taking $\eps$ appropriately small.
\end{proof}

The above lemma suffices as long as $k = o(n)$. For larger $k$ it is convenient to argue a little differently (and more simply).

\begin{lemma} \label{lem:Nk-bound-large-k}
Assume $c_1 + c'_1 = o(n)$ and $c+c' \leq n - \Omega(n)$. Then for $k\leq n/2$ we have
\[
  \E N_k \leq e^{-\Omega(k) + o(n)}.
\]
\end{lemma}
\begin{proof}
From Lemma~\ref{lem:fkx-bound} we have, for $0 < x < 1$,
\[
  f_k f'_k \leq x^{-2k} (1+x)^{c_1 + c'_1} (1+x^2)^{c + c'}.
\]
Take $x = (k/(n-k))^{1/2}$. We have
\[
  (1+x)^{c_1 + c'_1} \leq e^{x(c_1 + c'_1)} \leq e^{(2k/n)^{1/2} (c_1 + c'_1)}
\]
and
\[
  (1+x^2)^{c+c'} = (1+x^2)^n (1-k/n)^{n - c - c'} \leq (1+x^2)^n e^{-k(n - c- c')/n}.
\]
Hence
\[
  f_k f'_k \leq e^{h(k/n)n} e^{(2k/n)^{1/2}(c_1 + c'_1) - k(n-c-c')/n},
\]
and
\[
  \E N_k \leq \frac{f_k f'_k}{\binom{n}{k}} \leq e^{O(\log k) + (2k/n)^{1/2} (c_1 + c'_1) - k(n-c-c')/n}.
\]
Assuming $c_1+c'_1 = o(n)$ and $c + c' \leq (1-\eps) n$, this is bounded by
\[
  e^{O(\log k) + o((2k n)^{1/2}) - \eps k} = e^{o(n) - \eps k},
\]
as claimed.
\end{proof}

The following theorem follows by combining the previous two lemmas.

\begin{theorem} \label{thm:main-ENk-bound}
$\ $
\begin{enumerate}
    \item If $c_1 + c'_1 = o(n^{2/3})$ and $(c_1 + {c_2}^{1/2}) (c'_1 + {c'_2}^{1/2}) = o(n)$ then $\E N = o(1)$.
    \item If $c_1 + c'_1 = O(n^{2/3})$, $(c_1 + {c_2}^{1/2}) (c'_1 + {c'_2}^{1/2}) = O(n)$, and $c_2 + c'_2 = n - \Omega(n)$, then $\E N = O(1)$.
\end{enumerate}
\end{theorem}

\subsection{A Poisson approximation} \label{subsec:possion}

In the previous subsection we bounded the expectation $\E N$ of $N$. In this subsection we apply the method of moments to show that, under similar hypotheses, if $\E N$ is bounded then in fact
\[
  \P(N = 0) \approx e^{-\E N}.
\]
The method of moments depends on being able to prove moment estimates of the form
\[
  \E N(N-1) \cdots (N-m+1) \approx (\E N)^m.
\]
We therefore now turn our attention to the estimation of the moments and mixed moments of $(N_k)_{k\geq 1}$.

Recall from \eqref{eqn:ENk-formula} that
\[
  \E N_k = \sum_{\eqref{eqn:d-partition}, \eqref{eqn:d'-partition}} \frac{ \prod_{j=1}^k \binom{c_j}{d_j} \binom{c'_j}{d'_j}}{\binom{n}{k}} p(d_1, \dots, d_k; d'_1, \dots, d'_k).
\]
We first ask to what extent we have, for $k = O(1)$,
\[
  \E N_k \approx \sum_{\eqref{eqn:d-partition}, \eqref{eqn:d'-partition}} \frac{ \prod_{j=1}^k \frac{c_j^{d_j}}{d_j!} \frac{{c'_j}^{d'_j}}{d'_j!}}{\frac{n^k}{k!}} p(d_1, \dots, d_k; d'_1, \dots, d'_k).
\]
We have
\[
  \binom{c}{d} = (1 + O_d(1/c)) \frac{c^d}{d!},
\]
so the approximation is appropriate for the factors in which $c_j, c'_j$ are large, as well as the ones in which $d_j \leq 1$. Our contention is that the other terms do not contribute significantly, so overall the approximation is fine.

\begin{lemma}
We have
\[
  f_k \ll_k (c_1 + c^{1/2})^k.
\]
Moreover, for any particular index $j$ and any $t\geq 0$ we have
\[
  \sum_{\eqref{eqn:d-partition}, d_j \geq t} \binom{c_1}{d_1} \cdots \binom{c_k}{d_k} \ll_k \pfrac{c_j}{(c_1 + c^{1/2})^j}^t (c_1 + c^{1/2})^k.
\]
\end{lemma}
\begin{proof}
We revisit the calculation of the previous section, now armed with the restrictive hypothesis $k = O(1)$. From \eqref{fk},
\[
  f_k \asymp_k \max_\eqref{eqn:d-partition} c_1^{d_1} \cdots c_k^{d_k}.
\]
The maximum is not difficult to analyze: in fact the maximum over real $d_i$ is precisely
\[
  \max \{ c_1^k, c_2^{k/2}, \dots, c_k^{k/k}\}.
\]
We bound this crudely by
\[
  \max\{c_1^k, c^{k/2}\} \asymp_k (c_1 + c^{1/2})^k.
\]

Now consider the part of the sum in which $d_j \geq t$. By the same token we have
\begin{align*}
  \max_{\eqref{eqn:d-partition}, d_j \geq t} c_1^{d_1} \cdots c_k^{d_k}
  &= c_j^t \max_{d_1 1 + \cdots + d_k k = k - tj} c_1^{d_1} \cdots c_k^{d_k} \\
  &\ll_k c_j^t (c_1 + c^{1/2})^{k-tj}.
\end{align*}
This proves the lemma.
\end{proof}

\begin{lemma}
\begin{align*}
  \E N_k
  &= \sum_{\eqref{eqn:d-partition}, \eqref{eqn:d'-partition}} \frac{ \prod_{j=1}^k \frac{c_j^{d_j}}{d_j!} \frac{{c'_j}^{d'_j}}{d'_j!}}{\frac{n^k}{k!}} p(d_1, \dots, d_k; d'_1, \dots, d'_k) \\
  &\qquad + O_k\left( \frac1{\min\{c_1 + c^{1/2}, c'_1 + {c'}^{1/2}\}^{2/3}}
    \frac{((c_1 + c^{1/2})(c'_1 + {c'}^{1/2}))^k}{n^k}
  \right).
\end{align*}
\end{lemma}

\begin{proof}
Let $\delta$ be a parameter, and split the sum according to whether there is any $j$ such that $d_j \geq 2$ and $c_j \leq \delta (c_1 + c^{1/2})^j$, or $d'_j\geq 2$ and $c'_j \leq \delta (c'_1 + {c'}^{1/2})^j$. By the previous lemma, the part of the sum where there is some such $j$ is bounded by
\[
  O_k(\delta^2) (c_1 + c^{1/2})^k (c'_1 + {c'}^{1/2})^k.
\]
On the other hand, if $c_j \geq \delta (c_1 + c^{1/2})^j$ whenever $d_j \geq 2$ then we have
\[
  \binom{c_1}{d_1} \cdots \binom{c_k}{d_k}
  = \frac{c_1^{d_1}}{d_1!} \cdots \frac{c_k^{d_k}}{d_k!} \left(1 + O_k\pfrac{1}{\delta (c_1 + c^{1/2})} \right),
\]
and similarly for the primed variables. Thus the error in our approximation is bounded by
\[
  O_k\left(\delta^2 + \frac1{\delta (c_1 + c^{1/2})} + \frac1{\delta (c'_1 + {c'}^{1/2})} \right)
  \frac{(c_1 + c^{1/2})^k (c'_1 + {c'}^{1/2})^k}{n^k}.
\]
To get the best bound we take
\[
  \delta^3 = \min\{c_1 + c^{1/2}, c'_1 + {c'}^{1/2}\}^{-1}.
\]
This gives the claimed bound.
\end{proof}

There is a similar estimate for the mixed moment
\[
  \E[(N_1)_{m_1} \cdots (N_k)_{m_k}].
\]
Here we write $(x)_{m_i} = x(x-1) \cdots (x-m_i+1)$ for the ``falling factorial''. Let
\begin{align*}
  m &= m_1 + \cdots + m_k, \\
  M &= m_1 1 + \cdots + m_k k.
\end{align*}
Define $k_1, \dots, k_m$ by
\[
  (k_1, \dots, k_m) = (1^{m_1}, \dots, k^{m_k})
\]
(i.e., $k_i = 1$ for $i\leq m_1$, etc). The product
\[
  (N_1)_{m_1} \cdots (N_k)_{m_k}
\]
counts $m$-tuples of distinct orbits of $\langle \pi, \pi'\rangle$, of which $m_1$ are $1$-sets, $m_2$ are $2$-sets, etc., all these sets being \emph{disjoint} (this is why we count orbits rather than fixed sets). To choose such sets we have to choose, for every such size $k_i$, $d_{i1}$ $1$-cycles, $d_{i2}$ $2$-cycles, $\dots$ of $\pi$, and likewise $d'_{i1}$ $1$-cycles, $d'_{i2}$ $2$-cycles, $\dots$ of $\pi'$. This reasoning leads to the following estimate.

\begin{lemma}
\begin{align*}
  &\E[(N_1)_{m_1} \cdots (N_k)_{m_k}] \\
  &\qquad = \sum \prod_{i=1}^m \frac{\prod_{j=1}^{k} \frac{c_j^{d_{ij}}}{d_{ij}!}\frac{{c'_j}^{d'_{ij}}}{d'_{ij}!}}{\frac{n^{k_i}}{k_i!}} p(d_{i1}, \dots, d_{ik} ; d'_{i1}, \dots, d'_{ik}) \\
  &\qquad\qquad + O_M\left(
  \frac1{\min\{c_1 + c^{1/2}, c'_1 + {c'}^{1/2}\}^{2/3}}
  \frac{((c_1 + c^{1/2})(c'_1 + {c'}^{1/2}))^M}{n^M}
  \right).
\end{align*}
The sum runs over all $(d_{ij})$, $(d'_{ij})$ such that
\[
    d_{i1} 1 + \cdots + d_{ik} k =
    d'_{i1} 1 + \cdots + d'_{ik} k = k_i \qquad (i \in \{1, \dots, m\}).
\]
\end{lemma}

\begin{proof}
We have\footnote{We are using a slightly nonstandard, but backwards-compatible, notation for multinomial coefficients: $\binom{n}{a, b} = \frac{n!}{a!b!(n-a-b)!}$, etc.}
\begin{align*}
  &\E[(N_1)_{m_1} \cdots (N_k)_{m_k}] \\
  &\qquad = \sum \frac{\prod_{j=1}^{k} \binom{c_j}{d_{1j}, \dots, d_{mj}} \binom{c'_j}{d'_{1j}, \dots, d'_{mj}}}
  {\binom{n}{k_1, \dots, k_m}}
  \prod_{i=1}^m p(d_{i1}, \dots, d_{ik} ; d'_{i1}, \dots, d'_{ik}).
\end{align*}
The rest of the proof is the same as in the previous lemma.
\end{proof}

\begin{theorem} \label{thm:poisson-approximation}
Let
\[
  B = \max \{(c_1 + c^{1/2}) (c'_1 + {c'}^{1/2}) / n, 1\}.
\]
Then
\[
  \E[(N_1)_{m_1} \cdots (N_k)_{m_k}] - (\E N_1)^{m_1} \cdots (\E N_k)^{m_k}
  \ll_M
  \frac{B^M}{\min\{c_1 + c^{1/2}, c'_1 + {c'}^{1/2}\}^{2/3}}.
\]
Thus also, for any $k,m$,
\[
\E[(N_{\leq k})_m] - (\E N_{\leq k})^m 
 \ll_{k,m}
  \frac{B^{km}}{\min\{c_1 + c^{1/2}, c'_1 + {c'}^{1/2}\}^{2/3}}
  ,
\]
where $N_{\leq k} = N_1 + \cdots + N_k$.
\end{theorem}

\begin{proof}
The first bound follows immediately from the previous two lemmas. For the second bound, we use
\[
  (N_{\leq k})_m = m! \sum_{m_1 + \cdots + m_k = m} \frac{(N_1)_{m_1}}{m_1!} \cdots \frac{(N_k)_{m_k}}{m_k!},
\]
and similarly
\[
  (\E N_{\leq k})^m = m! \sum_{m_1 + \cdots + m_k = m} \frac{(\E N_1)^{m_1}}{m_1!} \cdots \frac{(\E N_k)^{m_k}}{m_k!}.
\]
Note that the value of
$
  M = 1 m_1 + \cdots + k m_k
$
over
$
  m_1 + \cdots + m_k = m
$
ranges between $m$ and $km$. This proves the theorem.
\end{proof}

\begin{theorem}
Assume $c_1 + c'_1 = o(n^{2/3})$, $(c_1 + c_2^{1/2}) (c'_1 + {c_2'}^{1/2}) = O(n)$, and $c_2 + c'_2 = n - \Omega(n)$. Then 
\[
  \P(N = 0) = e^{-\E N} + o(1).
\]
\end{theorem}

\begin{proof}
Write
$
  N = N_{\leq k} + N_{>k},
$
where
\begin{align*}
    N_{\leq k} &= \sum_{j = 1}^k N_j, \\
    N_{>k} &= \sum_{j > k} N_j.
\end{align*}
By Lemmas~\ref{lem:Nk-bound-small-k} and \ref{lem:Nk-bound-large-k} from the previous subsection, there is a $k = O_\eps(1)$ such that $\E N_{>k}$ is smaller than $\eps$, and
\[
  \P(N = 0) \leq \P(N_{\leq k} = 0) \leq \P(N = 0) + \P(N_{>k} > 0),
\]
so the main thing to understand is $\E N_{\leq k}$. From Bonferroni's inequalities we have, for any $M\geq 0$,
\[
  1_{N_{\leq k}=0} = \sum_{m=0}^{M-1} (-1)^m \binom{N_{\leq k}}{m} + O\left(\binom{N_{\leq k}}{M}\right).
\]
Therefore, from the previous theorem,
\begin{align*}
  \P(N_{\leq k} = 0)
  &= \sum_{m=0}^{M-1} (-1)^m \frac{ \E [(N_{\leq k})_{m}]}{m!} + O \pfrac{\E[ (N_{\leq k})_M]}{M!} \\
  &= \sum_{m=0}^{M-1} (-1)^m \frac{ (\E N_{\leq k})^m}{m!}
  + O\pfrac{(\E N_{\leq k})^M}{M!} \\
  &\qquad + O_{k, M} \left( 
  \frac{B^{kM}}{\min\{c_1 + c^{1/2}, c'_1 + {c'}^{1/2}\}^{2/3}}\right) \\
  &= e^{-\E N_{\leq k}} 
  + O\pfrac{(\E N_{\leq k})^M}{M!} \\
  &\qquad + O_{k, M} \left( 
  \frac{B^{kM}}{\min\{c_1 + c^{1/2}, c'_1 + {c'}^{1/2}\}^{2/3}}\right).
\end{align*}
Since $\E N_{\leq k} = O_\eps(1)$ (by Lemma~\ref{lem:Nk-bound-small-k} again), we can choose $M = O_\eps(1)$ so that the first error term here is smaller than $\eps$. Note also $B = O(1)$ by hypothesis. Putting all this together, we have
\[
  \P(N = 0) = e^{-\E N} + O(\eps) + O_{\eps}\left(\min\{c_1 + c^{1/2}, c'_1 + {c'}^{1/2}\}^{-2/3} \right).
\]

Finally, note that the hypothesis $c_1 + c'_1 = o(n^{2/3})$ ensures that either
\[
  (c_1 + c^{1/2}) (c'_1 + {c'}^{1/2}) = o(n),
\]
in which case we are already done by Theorem~\ref{thm:main-ENk-bound}, or
\[
  \min\{c_1 + c^{1/2}, c'_1 + {c'}^{1/2}\} \gg n^{1/3}.
\]
Thus
\[
  \P(N = 0) = e^{-\E N} + o(1),
\]
as claimed.
\end{proof}

\begin{remark}
An argument along the same line shows that $N$ is asymptotically distributed as $\Poisson(\E N)$, and indeed for any fixed $k$ the random variables $N_1, \dots, N_k$ are asymptotically distributed as independent Poisson random variables.
\end{remark}

\section{Transitive subgroups} \label{sec:transitive-subgroups}

In the previous section we estimated the probability that $G = \langle \pi, \pi'\rangle$ is transitive. In this section we show that $G$ is almost surely not a transitive subgroup smaller than $A_n$. The proof is easier (modulo known results about primitive groups), and the bounds rather stronger.

The following rather trivial lemma (which was already used implicitly in the previous section) seems worth isolating.

\begin{lemma}\label{fixed_point_lemma}
Suppose $G$ acts transitively on a set $X$, and suppose $\pi_1, \pi_2 \in G$ have $k_1$ and $k_2$ fixed points in $X$, respectively. Draw $\sigma \in G$ uniformly at random. Then
\[
  \P(\fix(\pi_1) \cap \fix(\pi_2^\sigma) \neq \emptyset) \leq k_1 k_2/|X|.
\]
\end{lemma}
\begin{proof}
This is immediate from
\[
  \E|\fix(\pi_1) \cap \fix(\pi_2^\sigma)| = \E|\fix(\pi_1) \cap \fix(\pi_2)^\sigma| = k_1 k_2/|X|.\qedhere
\]
\end{proof}

The utility of the lemma is easy to explain. Most maximal subgroups of $S_n$ are given explicitly as the stabilizer of a point in some natural action. For each such subgroup $M$, we need to show that $\pi, \pi'$ are not simultaneously trapped in a conjugate of $M$, i.e., do not have a common fixed point in the given action. By the lemma it will suffice to bound the number of fixed points of $\pi$ and $\pi'$ (with more than a square-root saving).

\subsection{Imprimitive subgroups} \label{subsec:imprimitive-subgroups}

First consider imprimitive transitive subgroups. To show that $\pi, \pi'$ are not simultaneously trapped by any such subgroup, we need to show $\pi, \pi'$ do not simultaneously preserve a partition of $\Omega$ into $k$ blocks of size $m$ for some $m,k>1$ such that $n=mk$.

Let $X_k$ be the set of all $k$-(equi)partitions of $\Omega$, i.e., partitions
\begin{equation} \label{k-partition}
  \Omega = \Omega_1 \cup \cdots \cup \Omega_k
\end{equation}
such that $|\Omega_i| = n/k$ for each $i$. The order of the cells $\Omega_1, \dots, \Omega_k$ is not considered significant. We have
\[
  |X_k| = \frac{n!}{k! (n/k)!^k}.
\]
Note that $S_n$ acts transitively on $X_k$, and $\langle \pi, \pi'\rangle$ preserves a $k$-partition if and only if $\pi$ and $\pi'$ have a common fixed point in $X_k$. By Lemma~\ref{fixed_point_lemma} it suffices to bound $|\fix_{X_k}(\pi)|$.

\begin{remark}
The ``dual'' problem of estimating the number of $\pi$ fixing at least one $k$-partition was considered in Diaconis--Fulman--Guralnick~\cite[Sections~5~and~6]{diaconis--fulman--guralnick} and Eberhard--Ford--Koukoulopoulos~\cite[Theorem~1.2]{eberhard--ford--koukoulopoulos}.
\end{remark}

\begin{lemma} \label{imprimitive_small_k_bound}
Suppose $\pi \in S_n$ has $c$ cycles. Then $|\fix_{X_k}(\pi)| \leq k^c$.
\end{lemma}
\begin{proof}
Fixed points in $X_k$ correspond to $k$-partitions \eqref{k-partition} preserved by $\pi$, i.e., such that for some $\tau\in S_k$ we have
\begin{equation} \label{tau-condition}
  \Omega_i^\pi = \Omega_{i^\tau} \qquad (i \in \{1, \dots, k\}).
\end{equation}
Fix $\tau \in S_k$, and let us count the \emph{ordered} partitions satisfying \eqref{tau-condition}. For each cycle of $\pi$, pick a base point. The position of the base point in the partition $\{\Omega_1, \dots, \Omega_k\}$, together with $\tau$, determines the position of every other point in the cycle. Thus there are at most $k$ choices for how to partition the cycle, and therefore at most $k^c$ choices for the partition.

To deduce a bound for $|\fix_{X_k}(\pi)|$ we need to (1) sum over all possibilities for $\tau \in S_k$, and (2) divide by $k!$ because the order of the cells in a partition is not important, so we just get $k^c$ again.
\end{proof}

We need a stronger bound when $k$ is very large, say when $k = n/2$ or $k=n/3$. In this regime we can use the following lemma.

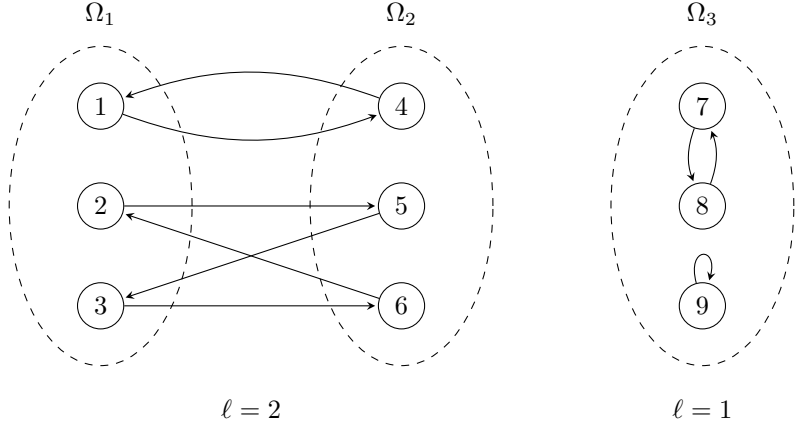
\begin{figure}
    \centering
    \input{tikz-invariant-partition.tex}
    \caption{A $\pi$-invariant partition with $m=3$, and the data indicated in the proof of Lemma~\ref{imprimitive_large_k_bound}}
    \label{fig:invariant-partition}
\end{figure}

\begin{lemma}\label{imprimitive_large_k_bound}
$|\fix_{X_{n/m}}(\pi)| \leq m^{O(n)} c^{(1-1/m)c}.$
\end{lemma}
\begin{proof}
A $\pi$-invariant $n/m$-partition is determined by the following data (see Figure~\ref{fig:invariant-partition}):
\begin{enumerate}
    \item Each $j$-cycle of $\pi$ must induce on the partition a cycle of length $\ell = j/d$ for some $d \leq m$. Label each cycle of $\pi$ by this integer $\ell$.
    \item Next, organize the cycles of $\pi$ into piles of cycles of a common label so that in any pile with label $\ell$, the sum of the lengths of the cycles is exactly $\ell m$.
    \item Finally, within each pile, determine how the cycles should be aligned.
\end{enumerate}
Clearly there are at most $m^c$ options for step 1. In step 3, we have at most $j$ options for each $j$-cycle, so all together we have at most
\[
  \prod_{j=1}^n j^{c_j} \leq \prod_{j=1}^n e^{j c_j} =  e^n
\]
options. The number of options in step 2 is bounded crudely by the number of ways of partitioning the $c$ cycles into piles of size at most $m$, which is at most
\[
  \frac{c^c}{\floor{c/m}!} \leq O(m)^{c/m} c^{c - c/m}.\qedhere
\]
\end{proof}

\begin{lemma} \label{lem:imprimitive-case}
Assume that $c + c' \leq (1 - \delta)n$ for some $\delta \geq \omega(\log \log n / \log n)$ (e.g., a constant). Then the probability that $\langle \pi, \pi'\rangle$ is contained in an imprimitive transitive subgroup is bounded by $2^{-\delta n + o(n)}$.
\end{lemma}
\begin{proof}
By Lemmas~\ref{fixed_point_lemma} and \ref{imprimitive_small_k_bound}, the probability that $\pi$ and $\pi'$ share a fixed point in $X_k$ is bounded by
\begin{equation}\label{imprim_prob_1}
  \frac{k^{c + c'}}{|X_k|} = \frac{k^{c + c'} k! m!^k}{n!} \leq \frac{k^{c + 
  c' + k} m!^k}{n!}.
\end{equation}
Using Stirling's approximation we have
\[
  \frac{m!^k}{n!}
  \asymp \frac{O(m)^{k/2} (m/e)^n}{n^{1/2} (n/e)^n} = \frac{O(1)^k n^{k/2 - 1/2}}{k^{k/2+n}}.
\]
Hence \eqref{imprim_prob_1} is bounded by
\[
  O(1)^k k^{c + c' +  k/2 - n} n^{k/2 - 1/2}.
\]
This bound is log-convex in $k$, so its maximum in the range $2 \leq k \leq n/\log n$ occurs at one of the end points. At $k=2$ we get
\[
  O(2^{c + c' - n} n^{1/2}),
\]
and at $k \asymp n / \log n$ we get
\[
  O(1)^n (n / \log n)^{c + c' - n} = (\log n)^{O(n)} n^{c+c' - n}.
\]

For $k \geq n / \log n$ we use Lemma~\ref{imprimitive_large_k_bound} instead of Lemma~\ref{imprimitive_small_k_bound}. Let $m = n/k$. Repeating the same calculation, we find that the probability that $\pi$ and $\pi'$ share a fixed point is bounded by
\begin{align*}
  m^{O(n)} n^{(1-1/m)(c + c')} k^{k/2 - n} n^{k/2-1/2}
  &\leq m^{O(n)} n^{(1-1/m)(c + c' - n)}.
\end{align*}

We obtain a bound for the probability that $\langle \pi, \pi'\rangle$ is imprimitive by summing over $k$. The result is
\[
  O(2^{c + c' - n} n^{3/2}) + (\log n)^{O(n)} n^{c + c' - n}.
\]
If $c + c' \leq (1-\delta)n$ (where $\delta \geq \omega(\log \log n / \log n)$) then the first term dominates, and the lemma follows.
\end{proof}

\begin{remark}
The lemma is likely not sharp, but improving it would require more careful methods. For example, suppose $\pi$ and $\pi'$ both have cycle type $2^{n/2}$, so that $c + c' = n$. As in Lemma~\ref{lem:p(k)}, there is a bijection between such pairs $(\pi, \pi')$ and permutations $\tau \in S_n$ having only even cycles, and $\langle \pi, \pi'\rangle$ preserves a $k$-partition if and only if $\tau$ does. This problem is similar to the one considered by {\L}uczak and Pyber: see~\cite{luczak-pyber}.
\end{remark}

\subsection{Primitive subgroups} \label{subsec:primitive-subgroups}

Finally we must consider the possibility that $\pi$ and $\pi'$ get trapped in a primitive subgroup. To some extent we could carry on as we have done for intransitive and imprimitive subgroups, since many families of primitive subgroups of $S_n$ are explicitly given as the stabilizer of a point in some action,\footnote{This is more or less the content of the O'Nan--Scott theorem (see~\cite{lps-onan-scott-thm} for an exposition and proof).} but at some point we will need to use some deeper knowledge about simple groups. Fortunately, at this point there is a convenient sledgehammer: primitive subgroups are few and small, while every conjugacy class with few cycles is big.

\begin{lemma}
Assume $\pi$ has $c$ cycles. Then $|\pi^{S_n}| \geq n!/n^c$.
\end{lemma}
\begin{proof}
We have
\[
  |\pi^{S_n}|
  = \frac{n!}{\prod_{j=1}^n j^{c_j} c_j!}
  \geq \frac{n!}{\prod_{j=1}^n (j c_j)^{c_j}}
  \geq \frac{n!}{n^c}. \qedhere
\]
\end{proof}

\begin{lemma} \label{lem:f_M-bound}
Let $M$ be a primitive maximal subgroup of $S_n$, other than $S_n$ or $A_n$. Let $f_M$ be the number of conjugates of $M$ containing $\pi$ (equivalently, the number of fixed points of $\pi$ in the action of $S_n$ on the conjugates of $M$). Then
\[
  f_M \leq n^c.
\]
\end{lemma}
\begin{proof}
The probability that a random conjugate of $M$ contains $\pi$ is the same as the probability that a random conjugate of $\pi$ is contained in $M$, so
\[
  \frac{f_M}{|S_n:M|} = \frac{|\pi^{S_n} \cap M|}{|\pi^{S_n}|}.
\]
(That $|S_n:M|$ is the number of conjugates of $M$ follows from maximality.) Therefore, by the previous lemma,
\[
  f_M = \frac{n! |\pi^{S_n} \cap M|}{|M| |\pi^{S_n}|} \leq \frac{n!}{|\pi^{S_n}|} \leq n^c.\qedhere
\]
\end{proof}

\begin{lemma} \label{lem:primitive-case}
The probability that $\langle \pi, \pi'\rangle$ is contained in a primitive subgroup other than $S_n$ or $A_n$ is bounded by $n^{c + c' + \sqrt{n}}/n!$.

In particular, if $c+c' \leq (1- \delta)n$ then this probability is bounded by $n^{-\delta n} e^{O(n)}$.
\end{lemma}
\begin{proof}
Let $M$ be a maximal primitive subgroup other than $S_n$ or $A_n$. By Lemma~\ref{fixed_point_lemma} again, along with the previous lemma, the probability that $\pi, \pi'$ are both contained in a conjugate of $M$ is bounded by
\[
  \frac{f_M f'_M}{|S_n : M|} \leq \frac{|M| n^{c + c'}}{n!}.
\]

Now as in Babai~\cite[Lemma~2.5 and the proof of Theorem~1.4]{babai_prob_gen}, the sum of $|M|$ over all conjugacy classes of primitive maximal subgroups of $S_n$ (other than $S_n$ and $A_n$) is bounded by $n^{\sqrt{n}}$. This proves the lemma.
\end{proof}

\begin{remark}
There do exist permutations $\pi$ such that $\pi$ and a random conjugate are trapped in a common primitive subgroup with positive probability. For example, suppose $n = k^2$, and consider $S_k \wr S_2$ acting on the cartesian square $\{1, \dots, k\}^2$. This gives us a map $S_k \wr S_2 \to S_n$ whose image is a primitive subgroup $M$ of order $2 k!^2$, and different conjugates of $M$ correspond to different labelling maps $\{1, \dots, k\}^2 \to \{1, \dots, n\}$. Let $\pi$ be the element of $M$ which swaps the first two rows, and let $\pi'$ be the element which swaps the third and fourth rows. Then $\pi$ and $\pi'$ are elements of cycle type $2^k$ with disjoint supports, and conversely any two elements of cycle type $2^k$ with disjoint supports are contained in a conjugate of $M$. Note that two random elements of cycle type $2^k$ have disjoint supports with probability $\sim e^{-4}$.

This may be essentially the only counterexample, however. We formalize this in the following conjecture, which we have so far not been able to prove.
\end{remark}

\begin{conjecture}
Let $\calC, \calC'\subset S_n$ be fixed nontrivial conjugacy classes, and let $\pi\in\calC$ and $\pi'\in\calC'$ be random. Then almost surely $\langle \pi, \pi'\rangle$ is not contained in any primitive subgroup apart from $S_n$, $A_n$, and any conjugate of $S_{\sqrt{n}} \wr S_2$ or $A_{\sqrt{n}} \wr S_2$.
\end{conjecture}

In any case, we have now proved under the hypotheses of Theorem~\ref{main-thm} that $G = \langle \pi, \pi'\rangle$ is almost surely not a transitive subgroup smaller than $A_n$, which completes the proof.

\section{Generators with a given order}\label{sec:fixed-order}
In this section we will apply Theorem~\ref{main-thm} to deduce some results concerning random generation of $S_n$ under order constraints. Let $\ordn=\{\ord \pi : \pi \in S_n\}$. For $m \in \ordn$, we are interested in the probability that two random elements of order $m$ generate at least $A_n$.

We may think of drawing a random $\pi\in S_n$ of order $m$ in two stages: first we pick a conjugacy class $\calC$ of elements of order $m$ with probability proportional to $|\calC|$, then we pick $\pi \in \calC$ uniformly. As a consequence of this and Theorem~\ref{main-thm}, we have the following criteria:

\begin{enumerate}
\item Assume almost all elements of order $m$ have $o(n^{1/2})$ fixed points and $o(n)$ $2$-cycles. Then two random elements of order $m$ almost surely generate at least $A_n$.
\item Assume almost all elements of order $m$ have $O(n^{1/2})$ fixed points and ${n/2-\Omega(n)}$ $2$-cycles. Then two random elements of order $m$ generate at least $A_n$ with probability bounded away from zero.
\end{enumerate}
Clearly, the converses in Theorem~\ref{main-thm} also apply. For instance, if a positive proportion of elements of order $m$ have $\Omega(n^{1/2})$ fixed points, the probability of generating at least $A_n$ will be bounded away from $1$.

Therefore, our business in this section is to understand which conditions on $m$ ensure that a random element of order $m$ has few fixed points and few $2$-cycles. This condition is certainly not automatic, as the following examples show.

\begin{example}
    The case $m = 2$ is rather special. A random permutation of order $2$ has $\sim n^{1/2}$ fixed points (see~\cite[Proposition~IX.19]{flajolet-sedgewick}) so random pairs of elements of order $2$ do not generate. In fact, since any two elements of order $2$ generate a dihedral group, indeed no two elements of order $2$ generate $A_n$.
\end{example}

\begin{example}
    Suppose $m  = p$ for some prime $p \sim 3n/4$. Then any element of order $m$ has $\Omega(n)$ fixed points, so random elements of order $m$ almost surely do not generate.
\end{example}

\begin{example}
    Similarly, suppose $m = 2p$ for some $p \sim 3n/4$. Then all $\pi$ of order $m$ have one $p$-cycle and either $\Omega(n)$ fixed points or $\Omega(n)$ $2$-cycles. It is not hard to see that random elements of order $m$ have $\Theta(n^{1/2})$ fixed points, so by Theorem~\ref{main-thm} random pairs of elements of order $m$ generate with probability bounded away from $0$ and from $1$.
\end{example}

\begin{example}
    There are also examples with much larger $m$. Let $p_1 < \cdots < p_k$ be primes, let $m = p_1 \cdots p_k$, and let
    \[
      n = \sum_{i=1}^k p_i + p_1 - 1.
    \]
    Then any $\pi \in S_n$ has one $p_i$-cycle for each $i$ and $p_1 - 1$ fixed points. If $p_1$ is large compared to $n^{1/2}$ then random pairs of permutations of order $m$ almost surely do not generate.
\end{example}

In each of these examples, the arithmetic of $m$ guarantees that \emph{all} permutations of order $m$ have either many fixed points or many $2$-cycles. As a last resort, one might hope at least for a ``zero--one law'': perhaps as long as there is at least one permutation of order $m$ with few fixed points and few $2$-cycles, then almost all permutations of order $m$ are such. Our last example also dishes this hope.

\begin{example} Let $p$ and $q$ distinct primes of roughly the same size, $p < q$, let $n = pq + p - 1$, and let $m = pq$. There are just two ways that $\pi \in S_n$ can have order $m$:
    \begin{enumerate}
        \item $\pi$ might be a $pq$-cycle. Any such $\pi$ has $p-1 \sim n^{1/2}$ fixed points.
        \item $\pi$ has cycle type $p^\mu q^\nu$ for some $\mu, \nu \geq 1$. There is some $\pi$ of this type with no fixed points (by the Frobenius postage stamp problem).
    \end{enumerate}
    The proportion of elements of the first type is
    \[
      \frac1{p q (p-1)!} \approx \frac1{(n^{1/2})!}.
    \]
    The proportion of elements of the second type is at most
    \[
      \frac1{p^\mu \mu! q^\nu \nu! (n - \mu p - \nu q)!} \lesssim \frac1{(n^{1/2})!^2}.
    \]
    Summing over $\mu, \nu$ adds just another two factors of $n$, so almost all $\pi$ of order $m$ are $pq$-cycles. Therefore the probability that two random elements of order $m$ generate is bounded away from $1$.
\end{example}

Having moderated our expectations, we will prove the following positive results. Assume $m \in \ordn$.
\begin{enumerate}
    \item If $m$ has a divisor $d$ in the range $3 \leq d \leq o(n^{1/2})$, then two random elements of order $m$ almost surely generate.
    \item If $m$ is even and there is at least one $\pi \in S_n$ with $o(n^{1/2})$ fixed points and $o(n)$ $2$-cycles, then two random elements of order $m$ almost surely generate.
    \item If $m$ is odd and there is at least one $\pi \in S_n$ with $O(n^{1/2})$ fixed points, then two random elements of order $m$ generate with positive probability.
    \item If $m > 2$ is even, then two random elements of order $m$ generate with positive probability.
\end{enumerate}
The first two of these points are the content of Theorem~\ref{thm:fixed-order-almost-sure}; the latter two are the content of Theorem~\ref{thm:fixed-order-positive-prob}.

\subsection{Theorem~\ref{thm:fixed-order-almost-sure}: Almost-sure generation}

We will use a simple computation a couple of times: for convenience we isolate the statement.

\begin{lemma} \label{easy_computation}
Let $\calC$ be a conjugacy class of $S_n$, and let $2 \leq d \leq n$. Assume $c_1 \geq 2kd$ for some $k \geq 1$. Define $\calC '$ by replacing $kd$ fixed points by $k$ $d$-cycles, i.e., $c'_1 = c_1-kd$, $c'_d=c_d+k$, and $c'_j=c_j$ for every other $j$. Then
\[
 \frac{|\calC|}{|\calC '|} \leq  \frac{(c_d+k)^k d^k}{(kd)^{kd}} \leq  \left( \frac{n}{(kd)^d} \right)^k.
\]
\end{lemma}
\begin{proof}
\[
  \frac{|\calC|}{|\calC'|} = \frac{\prod_j c'_j ! j^{c'_j}}{\prod_j c_j! j^{c_j}} = \frac{(c_1 - kd)! (c_d+k)! d^k}{c_1! c_d!} \leq \frac{(c_d + k)^k d^k}{(kd)^k} . \qedhere
\]
\end{proof}

To prove the first part of Theorem~\ref{thm:fixed-order-almost-sure} we must show that if $m$ has a small divisor (other than $2$) then almost all $\pi$ of order $m$ have $o(n^{1/2})$ fixed points and $o(n)$ $2$-cycles. This is the content of the following lemma.

\begin{lemma} \label{good_m_calculation}
Assume $m$ has a divisor $d$ such that $3 \leq d \leq o(n^{1/2})$. Then almost all permutations of order $m$ have $o(n^{1/2})$ fixed points and $o(n)$ $2$-cycles.
\end{lemma}
\begin{proof}
Let $k$ be such that $n^{1/2}/ \log n \leq kd \leq o(n^{1/2})$. We will show that almost all permutations of order $m$ have at most $2kd = o(n^{1/2})$ fixed points. If a conjugacy class $\calC$ is made of elements with $c_1 \geq 2kd$, define $\calC '$ as in the previous lemma. We have
\[
  \frac{|\calC|}{|\calC'|} \leq \left( n^{-1/2+o(1)}\right)^k =o(1).
\]
Therefore we have injectively associated to every conjugacy class of permutations of order $m$ with $c_1 \geq 2kd$ some much larger conjugacy class of permutations of order $m$, so it follows that almost all $\pi \in S_n$ of order $m$ have fewer than $2kd = o(n^{1/2})$ fixed points.

The proof for $2$-cycles is the same, starting with $k$ such that $n/ \log n \leq kd \leq o(n)$. If $c_2 \geq 2kd$, we define $c'_2 = c_2-kd$ and $c'_d=c_d+2k$, and a similar computation concludes the proof.
\end{proof}

Assume now $m$ is even. Then we can prove the zero--one law suggested earlier.

\begin{lemma}\label{good_m_calculation_2}
Assume $m$ is even, and assume there is an element of order $m$ having $o(n^{1/2})$ fixed points and $o(n)$ $2$-cycles. Then almost all elements of order $m$ have the same property.
\end{lemma}
\begin{proof}
Applying Lemma~\ref{easy_computation} with $d=2$, we get
\[
  \frac{|\calC|}{|\calC '|} \leq \left(\frac{c_2+k}{2k^2}\right)^k.
\]
This will be small unless $c_2 \geq k^2$. This means that the number of permutations of order $m$ with at least $4k$ fixed points is dominated by the number of permutations of order $m$ with more than $k^2$ $2$-cycles. Note that the density of such permutations is bounded by $1/(k^2)!$.

On the other hand let $\calC_0$ be a conjugacy class of elements of order $m$ having $o(n^{1/2})$ fixed points and $o(n)$ $2$-cycles. Write $a_j$ for the number of $j$-cycles. If $m$ has any small odd divisor, then we conclude by Lemma~\ref{good_m_calculation}. Otherwise, the elements of $\calC_0$ have $o(n)$ cycles. Therefore
\[
  \frac{|\calC_0|}{n!} = \frac{1}{\prod a_j! j^{a_j}} \geq \frac{1}{\prod (a_j j)^{a_j}} \geq n^{-o(n)}.
\]
Therefore we can choose $k = o(n^{1/2})$ so that $1/(k^2)! = o(|\calC_0| / n!)$. It follows that $\calC_0$ is much bigger than the set of all permutations with more than $k^2$ $2$-cycles, so again we find that almost all $\pi \in S_n$ have at most $4k = o(n^{1/2})$ fixed points and at most $k^2 = o(n)$ $2$-cycles.
\end{proof}

This completes the proof of Theorem~\ref{thm:fixed-order-almost-sure}.

\subsection{Theorem~\ref{thm:fixed-order-positive-prob}: Positive-probability generation} Now we investigate the conditions under which two random elements of order $m$ generate with positive probability. We are able to give a complete characterization of this property.

\begin{lemma} Let $m \in \ordn$.
\begin{enumerate}
\item Assume $m$ is odd, and assume there exists a permutation of order $m$ having $O(n^{1/2})$ fixed points. Then almost all permutations of order $m$ have the same property.
\item Assume $m \neq 2$ is even. Then almost all permutations of order $m$ have $O(n^{1/2})$ fixed points and $n/2-\Omega(n)$ $2$-cycles.
\end{enumerate}
\end{lemma}
\begin{proof}
(1) Let $d$ denote the smallest nontrivial divisor of $m$. If $d = o(n^{1/2})$, we conclude by Lemma~\ref{good_m_calculation}. Assume then $d = \Omega(n^{1/2})$. As in the proof of Lemma~\ref{good_m_calculation_2}, the density of permutations having at least $bn^{1/2}$ fixed points (for some constant $b$) is at most $1/(bn^{1/2})!$. Moreover, if $\calC_0$ is a conjugacy class of elements of order $m$ having $O(n^{1/2})$ fixed points, then the elements of $\calC_0$ have $O(n^{1/2})$ cycles, and
\[
 \frac{|\calC_0|}{n!} \geq n^{-O(n^{1/2})}.
\]
Thus if $b$ is a sufficiently large constant then this dominates, so indeed almost all elements of order $m$ have $O(n^{1/2})$ fixed points.

(2) Let $d$ denote the smallest divisor of $m$ larger than $2$. Note that $d$ is either $4$ or prime. As above, we may assume $d = \Omega(n^{1/2})$. The argument of Lemma~\ref{good_m_calculation} (with $k=1$) shows that almost all elements of order $m$ have at most $2d$ $2$-cycles. On the other hand any element of order $m$ must have a cycle of length at least $d$. Thus almost all permutations have at most $\min \{2d, (n-d)/2\}$ $2$-cycles. This is at most $2n/5 = n/2-\Omega(n)$, so the statement on $2$-cycles is proved.

Regarding fixed points, we apply Lemma~\ref{easy_computation} with $d=2$ and $k=\floor{n^{1/2}}$, getting
\[
 \frac{|\calC|}{|\calC '|} \leq \left(\frac{n}{4\floor{n^{1/2}}^2}\right)^{\floor{n^{1/2}}} = o(1).
\]
This shows that almost all elements of $S_n$ of order $m$ have at most $4n^{1/2}$ fixed points, and the proof is concluded.
\end{proof}

Theorem~\ref{thm:fixed-order-positive-prob} follows immediately from this and Theorem~\ref{main-thm}.

\bibliography{refs}
\bibliographystyle{alpha}

\end{document}

%% file: tikz-bijection.tex
\begin{tikzpicture}[>=stealth,shorten >=1pt,auto]

\def \n {4}
\def \radius {3cm}

\foreach \x in {1,...,\n}{
    \node[draw, circle] (\x) at (\x*360/\n : \radius) {$\x$};
}

\path[->]
    (1) edge [bend right] node[left] {$\tau$} (2)
    (2) edge [bend right] node[left] {$\tau$} (3)
    (3) edge [bend right] node[right] {$\tau$} (4)
    (4) edge [bend right] node[right] {$\tau$} (1);

\path[<->]
    (1) edge node [right] {$\sigma$} (2)
    (2) edge node [right] {$\sigma'$} (3)
    (3) edge node [left] {$\sigma$} (4)
    (4) edge node [left] {$\sigma'$} (1);
\end{tikzpicture}

%% file: tikz-invariant-partition.tex
\begin{tikzpicture}[>=stealth,shorten >=1pt,auto,
  every node/.style={draw,circle,-},
  every fit/.style={ellipse,draw,inner sep=-4pt,text width=2cm,dashed}
]
\usetikzlibrary{chains,fit,shapes}

\begin{scope}[start chain=going below,node distance=7mm]
\foreach \i in {1,2,3}
  \node[on chain] (\i) {\i};
\end{scope}

\begin{scope}[xshift=4cm,start chain=going below,node distance=7mm]
\foreach \i in {4,5,6}
  \node[on chain] (\i) {\i};
\end{scope}

\begin{scope}[xshift=8cm,start chain=going below,node distance=7mm]
\foreach \i in {7,8,9}
  \node[on chain] (\i) {\i};
\end{scope}

\node [fit=(1) (3), label=above:$\Omega_1$, label={[xshift=2cm]below:{$\ell = 2$}}] {};
\node [fit=(4) (6), label=above:$\Omega_2$] {};
\node [fit=(7) (9), label=above:$\Omega_3$, label=below:{$\ell = 1$}] {};

\foreach \x/\y/\b in {
  1/4/20,
  4/1/20,
  2/5/0,
  5/3/0,
  3/6/0,
  6/2/0,
  7/8/20,
  8/7/20
  }
\path[->] (\x) edge [bend right=\b] (\y);

\path[->] (9) edge [loop above] (9);

\end{tikzpicture}